\documentclass[oneside,american,english]{amsart}
\usepackage[T1]{fontenc}
\usepackage[latin9]{inputenc}
\usepackage{mathtools}
\usepackage{amstext}
\usepackage{amsthm}

\makeatletter
\numberwithin{equation}{section}
\numberwithin{figure}{section}
\theoremstyle{plain}
\newtheorem{thm}{\protect\theoremname}[section]
\theoremstyle{plain}
\newtheorem{cor}[thm]{\protect\corollaryname}
\theoremstyle{plain}
\newtheorem{lem}[thm]{\protect\lemmaname}
\theoremstyle{plain}
\newtheorem{prop}[thm]{\protect\propositionname}
\theoremstyle{remark}
\newtheorem{claim}[thm]{\protect\claimname}

\usepackage{colortbl}
\usepackage{graphicx}
\usepackage{pgf,tikz,tkz-graph,subcaption}
\usepackage{breqn}

\usetikzlibrary{arrows,shapes}
\usetikzlibrary{decorations.pathreplacing}

\makeatother

\usepackage{babel}
\addto\captionsamerican{\renewcommand{\claimname}{Claim}}
\addto\captionsamerican{\renewcommand{\corollaryname}{Corollary}}
\addto\captionsamerican{\renewcommand{\lemmaname}{Lemma}}
\addto\captionsamerican{\renewcommand{\propositionname}{Proposition}}
\addto\captionsamerican{\renewcommand{\theoremname}{Theorem}}
\addto\captionsenglish{\renewcommand{\claimname}{Claim}}
\addto\captionsenglish{\renewcommand{\corollaryname}{Corollary}}
\addto\captionsenglish{\renewcommand{\lemmaname}{Lemma}}
\addto\captionsenglish{\renewcommand{\propositionname}{Proposition}}
\addto\captionsenglish{\renewcommand{\theoremname}{Theorem}}
\providecommand{\claimname}{Claim}
\providecommand{\corollaryname}{Corollary}
\providecommand{\lemmaname}{Lemma}
\providecommand{\propositionname}{Proposition}
\providecommand{\theoremname}{Theorem}

\begin{document}
\title{On the difference of mean subtree orders under edge contraction}
\author{Ruoyu Wang}
\address{Department of Mathematics, Uppsala University}
\begin{abstract}
Given a tree $T$ of order $n,$ one can contract any edge and obtain
a new tree $T^{*}$ of order $n-1.$ In 1983, Jamison made a conjecture
that the mean subtree order, i.e., the average order of all subtrees,
decreases at least $\frac{1}{3}$ in contracting an edge of a tree.
In 2023, Luo, Xu, Wagner and Wang proved the case when the edge to
be contracted is a pendant edge. In this article, we prove that the
conjecture is true in general. 
\end{abstract}

\maketitle

\section{Introduction and Main Result}

Let $T$ be a tree of order $n$ and $v$ be a vertex in $T.$ The
mean subtree order of $T,$ also known as the global mean of $T,$
denoted by $\mu_{T},$ is defined as the average order of all subtrees
of $T.$ Similarly, the local mean of $T$ at $v,$ denoted by $\mu_{T}(v),$
is defined as the average order of all subtrees of $T$ that contain
$v.$ 

Jamison initiated the study of the global and local mean of a tree
in the 1980s \cite{MR735190,MR762896}. These two papers not only
laid groundwork for research into this invariant, but also made in
total seven conjectures, which received attention and progress of
different degrees over the last two decades. Before this paper, five
of the seven conjectures were settled \cite{MR4282633,MR3213626,MR3982896,MR2595700,MR3433637},
and the two that remained open are
\begin{enumerate}
\item Caterpillar conjecture: Is the tree of maximum density of each order
a caterpillar?
\item Contraction conjecture: The global mean of a tree decreases by at
least $\frac{1}{3}$ in contracting an edge.
\end{enumerate}
This paper focuses on the contraction conjecture, which is put forward
in \cite{MR762896}. In the same paper, Jamison showed that in general
the global mean respects contraction monotonicity, namely, the global
mean decreases under edge contraction. It is then of interest to establish
a bound for this difference. In the case of a path, the global mean
decreases exactly by $\frac{1}{3}.$ Indeed, for $T=P_{n},$ $\mu_{T}=\frac{n+2}{3},$
and $T^{*}=P_{n-1},$ where $T^{*}$ is obtained from $T$ by contracting
any edge. This observation justifies the constant $\frac{1}{3}$ in
the contraction conjecture since the minimum global mean is always
attained by the path. Concerning an upper bound for the contraction
difference, Jamison showed in the same paper \cite{MR762896} that
contracting a pendant edge attached at the center of a path can lead
to a difference of up to roughly $\frac{|T|}{18}.$ 

In 2023, Luo, Xu, Wagner and Wang \cite{MR4563206} showed that if
the edge to be contracted lies on a pendant path, i.e., one end of
the edge connects to some leaf through only vertices of degree two,
then the difference under contraction is at least $\frac{1}{3},$
with equality if and only if the tree is a path. In this article,
we will prove that when the edge to be contracted is a general internal
edge, namely, an edge that does not lie on a pendant path, the difference
is strictly greater than $\frac{1}{3}.$ 
\begin{thm}
\label{thm:mainthm}Let $T$ be a tree and $e$ be an edge that does
not lie on a pendant path. Let $T^{*}$ be the tree formed by contracting
$e$ from $T.$ Then
\[
\mu_{T}-\mu_{T^{*}}>\frac{1}{3}.
\]
\end{thm}

Hence, we settle the contraction conjecture with an affirmative answer.
Note that the case of the difference being equal to $\frac{1}{3}$
is then characterized by contracting any edge in a path.
\begin{cor}
Let $T$ be a tree and $e$ be an edge in $T.$ Let $T^{*}$ be the
tree formed by contracting $e$ from $T.$ Then
\[
\mu_{T}-\mu_{T^{*}}\ge\frac{1}{3},
\]
with equality if and only if $T$ is a path.
\end{cor}

Two inequalities that we establish to facilitate the proof are also
interesting in their own right.

It is known that the local mean at some vertex is no less than the
global mean (\cite{MR735190}, Theorem 3.9). The first inequality
provides a lower bound on the difference between local mean and global
mean.

Let $N_{T}$ be the number of subtrees in $T$ and $R_{T}$ the total
cardinality of these subtrees. Similarly, $N_{v}$ is the number of
subtrees in $T$ that contains $v$ and $R_{v}$ is the total cardinality
of these subtrees.
\begin{lem}
Let $T$ be a tree and $v$ a vertex. Then $T$ satisfies
\[
3(N_{T}R_{v}-N_{v}R_{T})\ge N_{T}^{2}-N_{v}N_{T},
\]
which is equivalent to
\[
\mu_{T}(v)-\mu_{T}\ge\frac{1}{3}\left(\frac{N_{T}}{N_{v}}-1\right).
\]
\end{lem}

The second inequality establishes an upper bound for the global mean
of $T-v,$ i.e., the tree with vertex $v$ removed. Note that here
$T-v$ can be a forest, i.e., a union of several disjoint trees.
\begin{lem}
\label{lem: mu_T-v}Let $T$ be a tree and $v$ a vertex. Then $T$
satisfies
\[
N_{v}N_{T}-N_{v}^{2}+N_{T}-N_{v}\ge3(R_{T}-R_{v}),
\]
which is equivalent to
\[
\mu_{T-v}\coloneqq\frac{R_{T}-R_{v}}{N_{T}-N_{v}}\le\frac{N_{v}+1}{3}.
\]
\end{lem}

\medskip{}

\section{Auxiliary Inequalities}

We first state without proof four inequalities from earlier papers.
\begin{align}
N_{v}^{2}+N_{v} & \ge2R_{v},\label{eq:JamisonIneq}\\
N_{v}^{2}+N_{v}+N_{T} & \ge3R_{v}.\label{eq:StephanIneq}\\
R_{v} & \ge N_{T},\label{eq:R>N}\\
N_{v}N_{T}+2N_{T} & \ge3R_{T}.\label{eq:NN+2N>3R}
\end{align}

The first goes back to Jamison's \cite{MR735190}, Lemma 3.2(c). The
second is proved in \cite{MR4563206}, Lemma 3.1. Inequality (\ref{eq:StephanIneq})
can be seen as a refinement of (\ref{eq:JamisonIneq}), in the sense
that $R_{v}$ is in general greater than or equal to $N_{T}$ (by
(\ref{eq:R>N})) which makes (\ref{eq:StephanIneq}) tighter than
(\ref{eq:JamisonIneq}). Note that equality holds in both if and only
if $T$ is a path and $v$ is a leaf. The third and fourth inequalities
are also proved in \cite{MR4563206}, Lemma 2.2 and Corollary 4.3.
Note also that if one attaches a pendant vertex $v_{0}$ to any vertex
$v\in T$ and applies Lemma \ref{lem: mu_T-v} to $v_{0}$ in $T+v_{0},$
then one has $\mu_{T}\le\frac{N_{v}+2}{3}$ which is equivalent to
(\ref{eq:NN+2N>3R}). 

Apart from the inequalities above, we shall also establish several
inequalities for rooted trees to facilitate the proof. Two rooted
trees will be dealt with separately in some cases, one is $S_{3}$
with root at the center and the other is $P_{4}$ rooted at an internal
vertex. For the sake of simplicity, we call these two rooted trees
$P_{2,2}$ and $P_{2,3}.$ Indeed, an alternative way to construct
$P_{2,2}$ is to identify two end vertices of two instances of $P_{2},$
and let the identified vertex be the root. For $P_{2,3},$ one can
identify the end vertex of $P_{2}$ with that of $P_{3}$ and again
let the identified vertex be the root. 

\begin{figure}[h]
\captionsetup{justification=centering}
\begin{minipage}[b]{.45\linewidth} 
\begin{center}     
\begin{tikzpicture}     
	{          
	
	\draw[thick] (0.5,0.25)--(-0.5,0)--(0.5,-0.25);

	\draw[fill] (-0.5,0) circle (0.08);     
	\draw[fill] (0.5,0.25) circle (0.08);
	\draw[fill] (0.5,-0.25) circle (0.08);
	
	\node at (-0.8,0) {$v$};     

   	} 	
\end{tikzpicture}\\ 
\subcaption{$P_{2,2}$} 
\end{center} 
\end{minipage}\quad\begin{minipage}[b]{.45\linewidth}
\begin{center}     
\begin{tikzpicture}
{ 	
	\draw[thick] (0.9,0.25)--(0,0)--(0.7,-0.15);
	\draw[thick] (0.7,-0.15)--(1.4,-0.3);
   
	\draw[fill] (0,0) circle (0.08);     
	\draw[fill] (0.9,0.25) circle (0.08);
	\draw[fill] (0.7,-0.15) circle (0.08);
	\draw[fill] (1.4,-0.3) circle (0.08);
	
	\node at (-0.3,0) {$v$};
	
	} 	
\end{tikzpicture}\\ 
\subcaption{$P_{2,3}$} 
\end{center} 
\end{minipage} 

\caption{$P_{2,2}$ and $P_{2,3}$.}
\label{fig:P22P23} 
\end{figure}
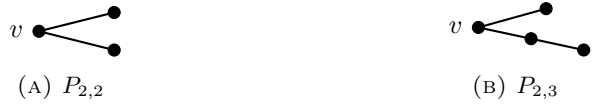 

Let $T_{v}$ be a rooted tree with $\deg(v)\ge2.$ Let $T_{v}\neq P_{2,2}$
or $P_{2,3}.$ Then the following two inequalities are true. 
\begin{align}
N_{v}^{2}+N_{v} & \ge3R_{v},\label{eq:N^2+N-3R}\\
N_{v}N_{T}+N_{T} & >3R_{T}.\label{eq:NN+N-3R}
\end{align}
The following two inequalities, which have been introduced before,
hold for any tree $T$ and any $v\in T.$ 
\begin{align}
3(N_{T}R_{v}-N_{v}R_{T}) & \ge N_{T}^{2}-N_{v}N_{T},\label{eq:3(NR-RN)}\\
N_{v}N_{T}-N_{v}^{2}+N_{T}-N_{v} & \ge3(R_{T}-R_{v}).\label{eq:3(R-R)}
\end{align}
Observe that (\ref{eq:StephanIneq}) and (\ref{eq:3(R-R)}) together
imply (\ref{eq:NN+2N>3R}) by taking the sum.

The inequalities above will be proved in Lemma \ref{lem:3(NR-RN)}
and \ref{lem:3(R-R)}--\ref{lem:NN+N>3R}.

\section{Proof of The Main Result}

A general edge that does not lie on a pendant path must lie in some
internal path whose both ends have degree greater than 2 and all vertices
in between have degree 2. Then contracting such an edge has the same
result as contracting any edge in that internal path. Thus, we consider
the case of contracting the edge that lies in one end of the path.
Let $v_{1}$ and $v_{2}$ be the two ends of such a path, and $v_{1}w_{1}$
be the edge to be contracted, as shown in Figure~\ref{fig:mainresult}.
Note that $v_{1}$ and $v_{2}$ both have degree no less than 3. Assume
that the path between $v_{1}$ and $v_{2}$ contains $l+1$ vertices,
$l\ge1.$

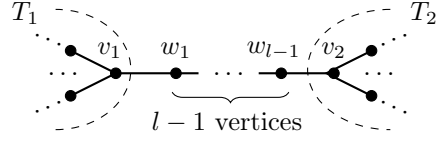
\begin{figure}[h]     
\begin{center}     
\begin{tikzpicture}     
	{          
	\draw[thick] (0,0)--(0.3,0);
	\draw[thick] (0,0)--(-0.8,0);
	\node at (0.7,0) {$\ldots$};  
	\draw[fill] (0,0) circle (0.07);     
	\draw[fill] (2.1,0) circle (0.07);
	\draw[thick] (2.6,-0.3)--(2,0)--(2.6,0.3);
	\draw[thick] (-1.4,-0.3)--(-0.8,0)--(-1.4,0.3);
	\draw[fill] (2.6,0.3) circle (0.07);
	\draw[fill] (2.6,-0.3) circle (0.07);
	\draw[fill] (-1.4,-0.3) circle (0.07);
	\draw[fill] (-1.4,0.3) circle (0.07);
	\draw[fill] (-0.8,0) circle (0.07);
	\draw[fill] (1.4,0) circle (0.07);
	
	\draw[thick] (1.1,0)--(2,0); 
	
	\draw[dashed] (2.85,0.85) arc (90:270:1.1cm and 0.8cm);
	\draw[dashed] (-1.6,-0.75) arc (-90:90:1cm and 0.8cm);

	\draw [decorate,decoration={brace,amplitude=4pt},xshift=0pt,yshift=0pt] (1.5,-0.25) -- (-0.05,-0.25) node [black,midway,xshift=-0.05cm, yshift=-0.4cm]{$l-1$ vertices};
	
	\node at (-0.88,0.3) {$v_1$};
	\node at (3.3,0.8) {$T_2$};
	\node at (-2,0.8) {$T_1$};
    \node at (0,0.3) {$w_1$};
	\node at (1.3,0.3) {$w_{l-1}$};
	\node at (2.1,0.3) {$v_2$};
	\node at (2.7,0) {$\ldots$};
	\node at (-1.5,0) {$\ldots$}; 

	\node[thick,rotate=23,anchor=center] at (2.9,0.45) {$\ldots$}; 
	\node[thick,rotate=-23,anchor=center] at (2.9,-0.45) {$\ldots$};

	\node[thick,rotate=-23,anchor=center] at (-1.7,0.45) {$\ldots$}; 
	\node[thick,rotate=23,anchor=center] at (-1.7,-0.45) {$\ldots$};
	     } 	
\end{tikzpicture}\\ 
\end{center} 
\caption{Contraction of the edge $v_1w_1$}
\label{fig:mainresult} 
\end{figure} 

Let $T_{1}$ ($T_{2}$) be the component containing $v_{1}$ ($v_{2}$)
in $T-v_{1}w_{1}$ ($T-w_{l-1}v_{2}$). Let $N_{1},R_{1},N_{L}$ and
$R_{L}$ be the number of subtrees in $T_{1}$ that contain $v_{1},$
the sum of the cardinalities of all such trees, the number of subtrees
in $T_{1}$ and their sum of the cardinalities. In the same way, we
define $N_{2},R_{2},N_{R}$ and $R_{R}.$ In general, for other cases,
we use $N$ for the number of subtrees involved, and $R$ for the
total cardinality. Then the mean subtree order can be expressed as
a function of $l,$ $\mu(l).$ 
\begin{align*}
\mu(l)= & \:\frac{R(\textrm{subtrees of }T_{2}+w_{1}\ldots w_{l-1})+R_{L}+R(\textrm{subtrees of \ensuremath{T} containing \ensuremath{v_{1}w_{1}}})}{N(\textrm{subtrees of }T_{2}+w_{1}\ldots w_{l-1})+N_{L}+N(\textrm{subtrees of \ensuremath{T} containing \ensuremath{v_{1}w_{1}}})},
\end{align*}
where $w_{1}\ldots w_{l-1}$ is the path from $w_{1}$ to $w_{l-1},$
consisting of $l-1$ vertices.

$N(\textrm{subtrees of }T_{2}+w_{1}\ldots w_{l-1})$ can be calculated
as follows,
\begin{align*}
 & \sum_{i=1}^{l-1}N(\textrm{subtrees of \ensuremath{T_{2}+w_{i}\ldots w_{l-1}} that contain \ensuremath{w_{i}}})+N_{R}=\sum_{i=1}^{l-1}(N_{2}+i)+N_{R}.
\end{align*}
Correspondingly, their total cardinality, $R(\textrm{subtrees of }T_{2}+w_{1}\ldots w_{l-1}),$
can be expressed as
\begin{align*}
 & \sum_{i=1}^{l-1}R(\textrm{subtrees of \ensuremath{T_{2}+w_{i}\ldots w_{l-1}} that contain \ensuremath{w_{i}}})+R_{R}\\
 & =\sum_{i=1}^{l-1}\left(R_{2}+iN_{2}+\frac{i(i+1)}{2}\right)+R_{R}.
\end{align*}
As for $N(\textrm{subtrees of \ensuremath{T} containing \ensuremath{v_{1}w_{1}}}),$
it simply equals the product
\begin{align*}
 & N_{1}\cdot N(\textrm{subtrees of }T_{2}+w_{1}\ldots w_{l-1}\textrm{ that contain }w_{1})=N_{1}(N_{2}+l-1).
\end{align*}
For their total cardinality, note that each subtree of $T_{1}$ that
contains $v_{1}$ is counted exactly $N(\textrm{subtrees of }T_{2}+w_{1}\ldots w_{l-1}\textrm{ that contain \ensuremath{w_{1}}})$
times in the total cardinality $R(\textrm{subtrees of \ensuremath{T} containing \ensuremath{v_{1}w_{1}}}),$
and similarly, each subtree of $T_{2}+w_{1}\ldots w_{l-1}$ that contains
$w_{1}$ is counted exactly $N_{1}$ times. Therefore, it can be calculated
as follows,
\begin{align*}
 & N(\textrm{subtrees of }T_{2}+w_{1}\ldots w_{l-1}\textrm{ that contain \ensuremath{w_{1}}})\cdot R_{1}\\
 & +N_{1}\cdot\left(R_{2}+(l-1)N_{2}+\frac{l(l-1)}{2}\right)\\
= & \:(N_{2}+l-1)R_{1}+N_{1}\left(R_{2}+(l-1)N_{2}+\frac{l(l-1)}{2}\right).
\end{align*}
Putting all parts together, we have
\begin{align*}
\mu(l)= & \:\frac{\frac{1}{6}l^{3}+(\frac{N_{1}}{2}+\frac{N_{2}}{2})l^{2}+Al+B}{\frac{1}{2}l^{2}+(N_{1}+N_{2}-\frac{1}{2})l+C},
\end{align*}
where
\begin{align*}
A= & \:R_{1}+R_{2}+N_{1}N_{2}-\frac{1}{2}N_{1}-\frac{1}{2}N_{2}-\frac{1}{6},\\
B= & \:N_{1}R_{2}+N_{2}R_{1}-N_{1}N_{2}+R_{L}+R_{R}-R_{1}-R_{2},\\
C= & \:N_{1}N_{2}+N_{L}+N_{R}-N_{1}-N_{2}.
\end{align*}
The difference of mean subtree orders under the contraction of $v_{1}w_{1}$
is simply
\begin{align*}
d(l)= & \:\mu(l)-\mu(l-1).
\end{align*}
Our goal is to show that $d(l)-\frac{1}{3}>0.$ This inequality, after
simplification, is equivalent to
\begin{align*}
[\textrm{I}]\cdot l(l-1)+[\textrm{II}]\cdot(2l-1)+[\textrm{III}] & >0,
\end{align*}
where
\begin{align*}
[\textrm{I}] & =\frac{C}{2}-\frac{3A}{2}+\frac{1}{2}(N_{1}+N_{2}+1)\left(N_{1}+N_{2}-\frac{1}{2}\right),\\{}
[\textrm{II}] & =\frac{C}{2}(N_{1}+N_{2}+1)-\frac{3B}{2},\\{}
[\textrm{III}] & =3AC-3B\left(N_{1}+N_{2}-\frac{1}{2}\right)-C^{2}.
\end{align*}

Since $l\ge1,$ the main result follows if $[\textrm{I}]>0,[\textrm{II}]>0$
and $[\textrm{III}]>0.$ We deal with each part separately. Also,
from this point of this section, we will use inequalities (\ref{eq:N^2+N-3R})
and (\ref{eq:NN+N-3R}) on $T_{1}$ and $T_{2}.$ Both inequalities
are established based on the assumption that the rooted tree is not
$P_{2,2}$ or $P_{2,3}.$ The cases in which at least one of $T_{1}$
and $T_{2}$ is either $P_{2,2}$ or $P_{2,3}$ will be dealt with
separately afterwards. Note that $T_{1}$ and $T_{2}$ not being $P_{2,2}$
or $P_{2,3}$ implies $N_{1},N_{2}>6,$ since $\deg_{T_{1}}(v_{1}),\deg_{T_{2}}(v_{2})\ge2.$

\subsection{Part {[}I{]}}
\begin{flushleft}
Replacing the expressions for $A$ and $C$ in $2[\textrm{I}]$ and
simplifying yields
\begin{align*}
2[\textrm{I}]= & \:(N_{1}N_{2}+N_{L}+N_{R}-N_{1}-N_{2})-3\left(R_{1}+R_{2}+N_{1}N_{2}-\frac{1}{2}N_{1}-\frac{1}{2}N_{2}-\frac{1}{6}\right)\\
 & +(N_{1}+N_{2}+1)\left(N_{1}+N_{2}-\frac{1}{2}\right)\\
= & \:N_{L}+N_{R}+(N_{1}^{2}+N_{2}^{2}+N_{1}+N_{2}-3R_{1}-R_{2}).
\end{align*}
Applying (\ref{eq:N^2+N-3R}) yields
\begin{align*}
2[\textrm{I}] & \ge N_{L}+N_{R}>0.
\end{align*}
\par\end{flushleft}

\subsection{Part {[}II{]}}

Replacing the expressions for $B$ and $C$ in 2{[}II{]} and simplifying
yields
\begin{align*}
2[\textrm{II}]= & \:(N_{1}N_{2}+N_{L}+N_{R}-N_{1}-N_{2})(N_{1}+N_{2}+1)\\
 & -3(N_{1}R_{2}+N_{2}R_{1}-N_{1}N_{2}+R_{L}+R_{R}-R_{1}-R_{2})\\
= & \:N_{2}(N_{1}^{2}+N_{1}-3R_{1})+N_{1}(N_{2}^{2}+N_{2}-3R_{2})+N_{1}N_{L}+N_{2}N_{L}+N_{1}N_{R}\\
 & +N_{2}N_{R}+N_{L}+N_{R}-N_{1}^{2}-N_{2}^{2}-N_{1}-N_{2}-3R_{L}-3R_{R}+3R_{1}+3R_{2}.
\end{align*}
By (\ref{eq:N^2+N-3R}), the two terms in parentheses are both non-negative.
Therefore,
\begin{align*}
2[\textrm{II}]\ge & \:N_{1}N_{L}+N_{2}N_{L}+N_{1}N_{R}+N_{2}N_{R}+N_{L}+N_{R}\\
 & -N_{1}^{2}-N_{2}^{2}-N_{1}-N_{2}-3R_{L}-3R_{R}+3R_{1}+3R_{2}\\
= & \:(N_{1}N_{L}+N_{L}-N_{1}^{2}-N_{1}-3R_{L}+3R_{1})+(N_{2}N_{R}+N_{R}-N_{2}^{2}-N_{2}-3R_{R}+3R_{2})\\
 & +N_{2}N_{L}+N_{1}N_{R}.
\end{align*}
Finally, by (\ref{eq:3(R-R)}), the terms in both parentheses are
non-negative. Hence,
\begin{align*}
2[\textrm{II}]\ge & \:N_{2}N_{L}+N_{1}N_{R}>0.
\end{align*}

\subsection{Part {[}III{]}}

Replacing the expressions $A,B$ and $C$ in part {[}III{]}, we have
\begin{align*}
[\textrm{III}]= & \:3\left(R_{1}+R_{2}+N_{1}N_{2}-\frac{1}{2}N_{1}-\frac{1}{2}N_{2}-\frac{1}{6}\right)(N_{1}N_{2}+N_{L}+N_{R}-N_{1}-N_{2})\\
 & -3(N_{1}R_{2}+N_{2}R_{1}-N_{1}N_{2}+R_{L}+R_{R}-R_{1}-R_{2})\left(N_{1}+N_{2}-\frac{1}{2}\right)\\
 & -(N_{1}N_{2}+N_{L}+N_{R}-N_{1}-N_{2})^{2}.
\end{align*}
Simplification leads to
\begin{align*}
[\textrm{III}]= & \:(2N_{1}^{2}N_{2}^{2}+3N_{L}R_{1}+3N_{R}R_{2}-3N_{1}^{2}R_{2}-3N_{2}^{2}R_{1}-3N_{1}R_{L}-3N_{2}R_{R})\\
 & +N_{1}N_{2}N_{L}+N_{1}N_{2}N_{R}+\frac{1}{2}N_{1}^{2}N_{2}+\frac{1}{2}N_{1}N_{2}^{2}-N_{1}N_{2}+3N_{L}R_{2}+3N_{R}R_{1}\\
 & -N_{L}^{2}-2N_{L}N_{R}-N_{R}^{2}+\frac{1}{2}N_{1}N_{L}+\frac{1}{2}N_{2}N_{L}+\frac{1}{2}N_{1}N_{R}+\frac{1}{2}N_{2}N_{R}\\
 & -\frac{1}{2}N_{L}-\frac{1}{2}N_{R}+\frac{1}{2}N_{1}^{2}+\frac{1}{2}N_{2}^{2}+\frac{1}{2}N_{1}+\frac{1}{2}N_{2}-3N_{1}R_{R}-3N_{2}R_{L}\\
 & +\frac{3}{2}N_{1}R_{2}+\frac{3}{2}N_{2}R_{1}+\frac{3}{2}R_{L}+\frac{3}{2}R_{R}-\frac{3}{2}R_{1}-\frac{3}{2}R_{2}.
\end{align*}
Next, we regroup the terms in parentheses as follows and apply inequalities
(\ref{eq:N^2+N-3R}) and (\ref{eq:3(NR-RN)}).
\begin{align*}
 & 2N_{1}^{2}N_{2}^{2}+3N_{L}R_{1}+3N_{R}R_{2}-3N_{1}^{2}R_{2}-3N_{2}^{2}R_{1}-3N_{1}R_{L}-3N_{2}R_{R}\\
= & \:N_{1}^{2}(N_{2}^{2}-3R_{2})+N_{2}^{2}(N_{1}^{2}-3R_{1})+3(N_{L}R_{1}-N_{1}R_{L})+3(N_{R}R_{2}-N_{2}R_{R})\\
\ge & \:N_{1}^{2}(-N_{2})+N_{2}^{2}(-N_{1})+N_{L}^{2}-N_{1}N_{L}+N_{R}^{2}-N_{2}N_{R}.
\end{align*}
We plug this into the main expression, then simplification and regrouping
terms yields
\begin{align*}
[\textrm{III}]\ge & \:\frac{1}{2}(N_{1}N_{2}N_{L}+N_{1}N_{2}N_{R}-N_{1}^{2}N_{2}-N_{1}N_{2}^{2}-N_{1}N_{L}-N_{2}N_{R}+N_{1}^{2}+N_{2}^{2})\\
 & +\frac{1}{2}N_{1}N_{2}N_{L}+\frac{1}{2}N_{1}N_{2}N_{R}-N_{1}N_{2}+3N_{L}R_{2}+3N_{R}R_{1}-2N_{L}N_{R}\\
 & +\frac{1}{2}N_{2}N_{L}+\frac{1}{2}N_{1}N_{R}-\frac{1}{2}N_{L}-\frac{1}{2}N_{R}+\frac{1}{2}N_{1}+\frac{1}{2}N_{2}-3N_{1}R_{R}-3N_{2}R_{L}\\
 & +\frac{3}{2}N_{1}R_{2}+\frac{3}{2}N_{2}R_{1}+\frac{3}{2}R_{L}+\frac{3}{2}R_{R}-\frac{3}{2}R_{1}-\frac{3}{2}R_{2}.
\end{align*}
Note that $N_{1}N_{2}N_{L}$ and $N_{1}N_{2}N_{R}$ are split into
two halves in order to proceed. Again, we deal with the terms in parentheses
separately. Regrouping and applying (\ref{eq:3(R-R)}) yields
\begin{align*}
 & \:\frac{1}{2}(N_{1}N_{2}N_{L}+N_{1}N_{2}N_{R}-N_{1}^{2}N_{2}-N_{1}N_{2}^{2}-N_{1}N_{L}-N_{2}N_{R}+N_{1}^{2}+N_{2}^{2})\\
= & \:\frac{1}{2}(N_{2}-1)(N_{1}N_{L}-N_{1}^{2})+\frac{1}{2}(N_{1}-1)(N_{2}N_{R}-N_{2}^{2})\\
\ge & \:\frac{1}{2}(N_{2}-1)(3(R_{L}-R_{1})-N_{L}+N_{1})+\frac{1}{2}(N_{1}-1)(3(R_{R}-R_{2})-N_{R}+N_{2}).
\end{align*}
Plugging this into the main expression, after simplification and cancellation
of terms, we have
\begin{align*}
[\textrm{III}]\ge & \:\frac{1}{2}N_{1}N_{2}N_{L}+\frac{1}{2}N_{1}N_{2}N_{R}+3N_{L}R_{2}+3N_{R}R_{1}-2N_{L}N_{R}-\frac{3}{2}N_{1}R_{R}-\frac{3}{2}N_{2}R_{L}\\
= & \:\frac{1}{2}N_{1}(N_{2}N_{R}-3R_{R})+\frac{1}{2}N_{2}(N_{1}N_{L}-3R_{L})+3N_{L}R_{2}+3N_{R}R_{1}-2N_{L}N_{R}.
\end{align*}
Finally, applying (\ref{eq:NN+N-3R}) to the terms in parentheses
yields

\begin{align*}
[\textrm{III}]\ge & \:\frac{1}{2}N_{1}(-N_{R})+\frac{1}{2}N_{2}(-N_{L})+3N_{L}R_{2}+3N_{R}R_{1}-2N_{L}N_{R}\\
= & \:3N_{L}R_{2}+3N_{R}R_{1}-2N_{L}N_{R}-\frac{1}{2}N_{1}N_{R}-\frac{1}{2}N_{2}N_{L}\\
\ge & \:\frac{3}{2}N_{L}R_{2}+\frac{3}{2}N_{R}R_{1}>0.
\end{align*}
The last step follows from $N_{L}\ge N_{1},R_{1}\ge N_{1},N_{R}\ge N_{2},R_{2}\ge N_{2}$
and (\ref{eq:R>N}).

\subsection{Cases when at least one of $T_{1}$ and $T_{2}$ is small}

The inequalities (\ref{eq:N^2+N-3R}) and (\ref{eq:NN+N-3R}) break
when the rooted tree is $P_{2,2}$ or $P_{2,3}.$ For $P_{2,2},$
we have $N_{v}=4,N_{T}=6,R_{v}=8$ and $R_{T}=10.$ For $P_{2,3},$
we have $N_{v}=6,N_{T}=10,R_{v}=15$ and $R_{T}=20.$ There are five
subcases to be dealt with in order to complete the proof. Without
loss of generality, assume $N_{1}\le N_{2}.$ Also, recall that $v_{1}$
and $v_{2}$ both have degree at least 3 (or have degree at least
2 when considered as root of $T_{1}$ and $T_{2}$).
\begin{itemize}
\item $T_{1}=P_{2,2},T_{2}\neq P_{2,2},P_{2,3}.$ So (\ref{eq:N^2+N-3R})
and (\ref{eq:NN+N-3R}) apply to $T_{2}$ only. For $T_{1},$ we have
$N_{1}=4,R_{1}=8,N_{L}=6,R_{L}=10,$ thus 
\begin{align*}
 & A=\frac{35}{6}+\frac{7}{2}N_{2}+R_{2},\\
 & B=2+4N_{2}+3R_{2}+R_{R},\\
 & C=2+3N_{2}+N_{R}.
\end{align*}
Plugging these into the three parts yields
\[
2[\textrm{I}]=N_{2}^{2}+N_{2}-3R_{2}+N_{R}+2\overset{(2.5)}{\geq}N_{R}+2>0,
\]
and
\begin{align*}
2[\textrm{II}]= & \:(N_{2}N_{R}+N_{R}-3R_{R})+3(N_{2}^{2}+N_{2}-3R_{2})+4N_{R}+2N_{2}+4,
\end{align*}
so applying (\ref{eq:N^2+N-3R}) and (\ref{eq:NN+N-3R}) yields
\begin{align*}
2[\textrm{II}]\ge & \:4N_{R}+2N_{2}+4>0.
\end{align*}
At last,
\begin{align*}
2[\textrm{III}]= & \:2(3N_{R}R_{2}-3N_{2}R_{R}-N_{R}^{2})+9N_{2}N_{R}+21N_{2}^{2}\\
 & +27N_{R}+27N_{2}-21R_{R}-51R_{2}+20,
\end{align*}
so applying (\ref{eq:3(NR-RN)}) to the terms in parentheses yields
\begin{align*}
2[\textrm{III}]\ge & \:7N_{2}N_{R}+21N_{2}^{2}+27N_{R}+27N_{2}-21R_{R}-51R_{2}+20\\
= & \:(17N_{2}^{2}+17N_{2}-51R_{2})+(7N_{2}N_{R}+7N_{R}-21R_{R})\\
 & +4N_{2}^{2}+10N_{2}+20N_{R}+20\\
\ge & \:4N_{2}^{2}+10N_{2}+20N_{R}+20>0,
\end{align*}
where the last step used (\ref{eq:N^2+N-3R}) and (\ref{eq:NN+N-3R})
for the terms in parentheses.\\
\item $T_{1}=P_{2,3},T_{2}\neq P_{2,2},P_{2,3}.$ So (\ref{eq:N^2+N-3R})
and (\ref{eq:NN+N-3R}) apply to $T_{2}$ only. For $T_{1},$ we have
$N_{1}=6,R_{1}=15,N_{L}=10,R_{L}=20.$ Thus 
\begin{align*}
 & A=\frac{71}{6}+\frac{11}{2}N_{2}+R_{2},\\
 & B=5+9N_{2}+5R_{2}+R_{R},\\
 & C=4+5N_{2}+N_{R}.
\end{align*}
Plugging these into the three parts yields
\[
2[\textrm{I}]=N_{2}^{2}+N_{2}-3R_{2}+N_{R}+13\stackrel{(2.5)}{\ge}N_{R}+13>0,
\]
and
\begin{align*}
2[\textrm{II}]= & \:(N_{2}N_{R}+N_{R}-3R_{R})+5(N_{2}^{2}+N_{2}-3R_{2})+6N_{R}+7N_{2}+13,
\end{align*}
so applying (\ref{eq:N^2+N-3R}) and (\ref{eq:NN+N-3R}) yields
\begin{align*}
2[\textrm{II}]\ge & \:6N_{R}+7N_{2}+13>0.
\end{align*}
At last,
\begin{align*}
2[\textrm{III}]= & \:2(3N_{R}R_{2}-3N_{2}R_{R}-N_{R}^{2})+13N_{2}N_{R}+61N_{2}^{2}\\
 & +55N_{R}+80N_{2}-33R_{R}-141R_{2}+87,
\end{align*}
so applying (\ref{eq:3(NR-RN)}) to the terms in parentheses yields
\begin{align*}
2[\textrm{III}]\ge & \:11N_{2}N_{R}+61N_{2}^{2}+55N_{R}+80N_{2}-33R_{R}-141R_{2}+87\\
= & \:(47N_{2}^{2}+47N_{2}-141R_{2})+(11N_{2}N_{R}+11N_{R}-33R_{R})\\
 & +14N_{2}^{2}+33N_{2}+44N_{R}+87\\
\ge & \:14N_{2}^{2}+33N_{2}+44N_{R}+87>0,
\end{align*}
where the last step used (\ref{eq:N^2+N-3R}) and (\ref{eq:NN+N-3R})
for the terms in parentheses.\\
\item $T_{1}=T_{2}=P_{2,2}.$ We have $N_{1}=4,R_{1}=8,N_{L}=6,R_{L}=10,$
and the same for $N_{2},R_{2},N_{R},R_{R}.$ Then $A=\frac{167}{6},B=52,C=20,$
thus $[\textrm{I}]=2,[\textrm{II}]=12,[\textrm{III}]=100.$
\item $T_{1}=P_{2,2},T_{2}=P_{2,3}.$ We have $N_{1}=4,R_{1}=8,N_{L}=6,R_{L}=10,$
and $N_{2}=6,R_{2}=15,N_{R}=10,R_{R}=20.$ Then $A=\frac{251}{6},B=91,C=30,$
thus $[\textrm{I}]=\frac{9}{2},[\textrm{II}]=\frac{57}{2},[\textrm{III}]=\frac{543}{2}.$
\item $T_{1}=T_{2}=P_{2,3}.$ We have $N_{1}=6,R_{1}=15,N_{L}=10,R_{L}=20,$
and the same for $N_{2},R_{2},N_{R},R_{R}.$ Then $A=\frac{359}{6},B=154,C=44,$
thus $[\textrm{I}]=7,[\textrm{II}]=55,[\textrm{III}]=649.$
\end{itemize}
This concludes the proof of Theorem \ref{thm:mainthm}.\\

Now we know that the difference between global means is $\frac{1}{3}$
if and only if we contract an edge in a path, and for all other cases,
the difference is strictly greater than $\frac{1}{3}.$ A natural
question to ask at this stage is whether this lower bound is asymptotically
sharp for trees that are not paths. We answer this question in the
following proposition.

Let $T_{1}$ and $T_{2}$ be two rooted trees, and $T(T_{1},T_{2},l)$
be the tree constructed by connecting the two roots by a path of $l$
vertices.
\begin{prop}
Let $T_{1}$ and $T_{2}$ be two rooted trees and $T(l)=T(T_{1},T_{2},l).$
Then 
\[
\lim_{l\rightarrow\infty}\mu_{T(l)}-\mu_{T(l-1)}=\frac{1}{3}.
\]
\end{prop}

\begin{proof}
Using the same notations, $d(l)-\frac{1}{3}$ can be expressed as
follows,
\begin{align*}
\frac{\left(\left(\frac{D}{2}+\frac{1}{4}\right)D+\frac{C}{6}-\frac{A}{2}\right)l(l-1)+\left(\left(\frac{D}{2}+\frac{1}{4}\right)C-\frac{B}{2}\right)(2l-1)+AC-BD}{\frac{1}{4}l^{2}(l-1)^{2}+Dl(l-1)(2l-1)+(D^{2}+C)l(l-1)+DC(2l-1)+C^{2}+\frac{C}{2}},
\end{align*}
where $D=N_{1}+N_{2}-\frac{1}{2},$ and $A,B$ and $C$ are the same
as defined in the beginning of this section. Observe that $A,B,C$
and $D$ are all constants given fixed $T_{1}$ and $T_{2}.$ Therefore,
$d(l)-\frac{1}{3}$ is the quotient of polynomials of $l$ whose numerator
is quadratic and denominator is of degree four. The statement follows
after letting $l$ go to infinity.
\end{proof}
Intuitively, as the length of the path tends to infinity, the structures
at both ends become negligible and the tree becomes more path-like.
Thus one can expect the difference to converge to that obtained for
a path.\\

\section{Proof of Inequalities}

The four inequalities (\ref{eq:N^2+N-3R})--(\ref{eq:3(R-R)}) will
be proved in this section. We start with (\ref{eq:3(NR-RN)}) and
(\ref{eq:3(R-R)}).

For (\ref{eq:3(NR-RN)}), we use a two-step strategy. First, in the
case $\deg(\textrm{root})=1,$ we prove that if the subtree rooted
at the only neighbour of the root satisfies the inequality then so
does the entire rooted tree. Then, we prove that if two rooted trees
both satisfy the inequality, then their ``union'', obtained by identifying
their roots, also satisfies the inequality. We will write $T_{v}=T_{1}\overset{v}{\cup}T_{2}$
for this union, as shown in the figure below. At last, the general
conclusion follows by induction on the order of the tree.

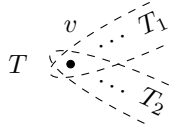
\begin{figure}[h]     
\begin{center}     
\begin{tikzpicture}     
	{          
	    
	\draw[fill] (0,0) circle (0.05);  

	\node[thick,rotate=25,anchor=center] at (0.55,0.3) {$\ldots$}; 
	\node[thick,rotate=-25,anchor=center] at (0.55,-0.3) {$\ldots$};

	\draw[dashed,rotate=25,anchor=center] (1.3,0.3) arc (90:270:1.6cm and 0.3cm);
	\draw[dashed,rotate=-25,anchor=center] (1.3,0.3) arc (90:270:1.6cm and 0.3cm);
	
    \node at (0,0.5) {$v$};
	\node at (-0.7,0) {$T$};
	\node[rotate=25,anchor=center] at (1.1,0.55) {$T_1$};
	\node[rotate=-25,anchor=center] at (1.1,-0.5) {$T_2$};   
	     } 	
\end{tikzpicture}\\ 
\end{center} 
\caption{$T$ is the union of $T_1$ and $T_2$.}
\label{fig:union} 
\end{figure} 
\begin{claim}
\label{claim:3(NR-RN)_DEG1}Let $T$ be a rooted tree with root at
$v$ and $\deg(v)=1.$ Let $v_{1}$ be the only vertex adjacent to
$v$ and $T_{1}=T-v.$ If $T_{1}$ satisfies (\ref{eq:3(NR-RN)}),
then so does $T.$
\end{claim}

\begin{proof}
Let $N_{v},R_{v},N_{T}$ and $R_{T}$ be the number of subtrees in
$T$ that contain $v,$ the sum of the cardinalities of all such trees,
the number of subtrees in $T$ and their sum of the cardinalities.
$N_{1},R_{1},N_{T_{1}}$ and $R_{T_{1}}$ are defined similarly with
respect to $T_{1}$ and $v_{1}.$

We start with
\begin{align*}
 & N_{v}=N_{1}+1,R_{v}=R_{1}+N_{1}+1,\\
 & N_{T}=N_{T_{1}}+N_{1}+1,R_{T}=R_{T_{1}}+R_{1}+N_{1}+1.
\end{align*}
By assumption, we have
\[
N_{1}N_{T_{1}}-N_{T_{1}}^{2}+3N_{T_{1}}R_{1}-3N_{1}R_{T_{1}}\ge0.
\]
Thus
\begin{align*}
 & N_{v}N_{T}-N_{T}^{2}+3N_{T}R_{v}-3N_{v}R_{T}\\
= & \:2N_{1}N_{T_{1}}+2N_{T_{1}}-N_{T_{1}}^{2}-3R_{T_{1}}+3(N_{T_{1}}R_{1}-N_{1}R_{T_{1}})\\
\ge & \:N_{1}N_{T_{1}}+2N_{T_{1}}-3R_{T_{1}}\\
\ge & \:0,
\end{align*}
where the last step uses (\ref{eq:NN+2N>3R}).
\end{proof}
\begin{claim}
\label{claim:3(NR-RN)_UNION}Let $T_{v}$ be a rooted tree with root
at $v$ and $\deg(v)\ge2.$ Let $T_{1}$ and $T_{2}$ be two non-singleton
subtrees also rooted at $v$ such that $T_{v}=T_{1}\overset{v}{\cup}T_{2}.$
If $T_{1}$ and $T_{2}$ both satisfy (\ref{eq:3(NR-RN)}), then so
does $T_{v}.$
\end{claim}

\begin{proof}
Using the same notations, we start with
\begin{align*}
 & N_{v}=N_{1}N_{2},R_{v}=N_{1}R_{2}+N_{2}R_{1}-N_{1}N_{2},\\
 & N_{T}=N_{1}N_{2}+N_{T_{1}}-N_{1}+N_{T_{2}}-N_{2},\\
 & R_{T}=N_{1}R_{2}+N_{2}R_{1}-N_{1}N_{2}+R_{T_{1}}-R_{1}+R_{T_{2}}-R_{2},
\end{align*}
and the assumption provides
\begin{align*}
3(N_{T_{1}}R_{1}-N_{1}R_{T_{1}}) & \ge N_{T_{1}}^{2}-N_{1}N_{T_{1}},\\
3(N_{T_{2}}R_{2}-N_{2}R_{T_{2}}) & \ge N_{T_{2}}^{2}-N_{2}N_{T_{2}}.
\end{align*}
Now
\begin{align*}
 & N_{v}N_{T}-N_{T}^{2}+3N_{T}R_{v}-3N_{v}R_{T}\\
= & -(N_{T_{1}}-N_{1}+N_{T_{2}}-N_{2})^{2}-4N_{1}N_{2}(N_{T_{1}}-N_{1})-4N_{1}N_{2}(N_{T_{2}}-N_{2})\\
 & +3\left((N_{T_{1}}-N_{1}+N_{T_{2}}-N_{2})(N_{1}R_{2}+N_{2}R_{1})-N_{1}N_{2}(R_{T_{1}}-R_{1}+R_{T_{2}}-R_{2})\right).
\end{align*}
Simplifying and regrouping the terms in the last line separately,
we have
\begin{align*}
 & 3((N_{T_{1}}-N_{1}+N_{T_{2}}-N_{2})(N_{1}R_{2}+N_{2}R_{1})-N_{1}N_{2}(R_{T_{1}}-R_{1}+R_{T_{2}}-R_{2}))\\
= & \:3N_{1}N_{T_{1}}R_{2}+3N_{2}N_{T_{2}}R_{1}-3N_{1}^{2}R_{2}-3N_{2}^{2}R_{1}\\
 & +3N_{1}(N_{T_{2}}R_{2}-N_{2}R_{T_{2}})+3N_{2}(N_{T_{1}}R_{1}-N_{1}R_{T_{1}})\\
\ge & \:3N_{1}N_{T_{1}}R_{2}+3N_{2}N_{T_{2}}R_{1}-3N_{1}^{2}R_{2}-3N_{2}^{2}R_{1}\\
 & +N_{1}(N_{T_{2}}^{2}-N_{2}N_{T_{2}})+N_{2}(N_{T_{1}}^{2}-N_{1}N_{T_{1}})\\
= & \:(N_{T_{1}}-N_{1})(3N_{1}R_{2}+N_{2}N_{T_{1}})+(N_{T_{2}}-N_{2})(3N_{2}R_{1}+N_{1}N_{T_{2}}),
\end{align*}
where the inequality follows from the assumption. Plugging this into
the entire expression, we have
\begin{align*}
 & N_{v}N_{T}-N_{T}^{2}+3N_{T}R_{v}-3N_{v}R_{T}\\
\ge & -(N_{T_{1}}-N_{1}+N_{T_{2}}-N_{2})^{2}-4N_{1}N_{2}(N_{T_{1}}-N_{1})-4N_{1}N_{2}(N_{T_{2}}-N_{2})\\
 & +(N_{T_{1}}-N_{1})(3N_{1}R_{2}+N_{2}N_{T_{1}})+(N_{T_{2}}-N_{2})(3N_{2}R_{1}+N_{1}N_{T_{2}})\\
= & \:(N_{T_{1}}-N_{1})(3N_{1}R_{2}+N_{2}N_{T_{1}}-4N_{1}N_{2}-N_{T_{1}}+N_{1}-N_{T_{2}}+N_{2})\\
 & +(N_{T_{2}}-N_{2})(3N_{2}R_{1}+N_{1}N_{T_{2}}-4N_{1}N_{2}-N_{T_{1}}+N_{1}-N_{T_{2}}+N_{2}).
\end{align*}
In the last step, as $N_{T_{1}}\ge N_{1}$ and $N_{T_{2}}\ge N_{2},$
the main inequality follows if 
\begin{align*}
3N_{1}R_{2}+N_{2}N_{T_{1}}-4N_{1}N_{2}-N_{T_{1}}+N_{1}-N_{T_{2}}+N_{2} & \ge0\\
3N_{2}R_{1}+N_{1}N_{T_{2}}-4N_{1}N_{2}-N_{T_{1}}+N_{1}-N_{T_{2}}+N_{2} & \ge0.
\end{align*}
Due to symmetry, it suffices to prove only the first one. By (\ref{eq:R>N}),
$R_{2}\ge N_{T_{2}},$ hence
\begin{align*}
\textrm{left hand side } & \ge3N_{1}N_{T_{2}}+N_{2}N_{T_{1}}-4N_{1}N_{2}-N_{T_{1}}+N_{1}-N_{T_{2}}+N_{2}\\
 & =(N_{T_{2}}-N_{2})(3N_{1}-1)+(N_{T_{1}}-N_{1})(N_{2}-1)\\
 & \ge0.
\end{align*}
\end{proof}
\begin{lem}
\label{lem:3(NR-RN)}Let $T$ be a rooted tree with root at $v.$
Then $T$ satisfies (\ref{eq:3(NR-RN)}).
\end{lem}

\begin{proof}
Let $N_{v}$ be the number of subtrees in $T$ containing $v.$ Let
$v_{1},\ldots,v_{d}$ be all the vertices adjacent to $v,$ with integer
$d\ge1.$ Let $T_{1},\ldots,T_{d}$ be the rooted subtrees with roots
at $v_{1},\ldots,v_{d}.$ We prove the statement by induction on the
order of $T.$ For $T$ being a singleton tree, both sides of the
inequality are equal to 3. Now assume (\ref{eq:3(NR-RN)}) holds for
$T_{1},\ldots,T_{d}.$ Then by Claim \ref{claim:3(NR-RN)_DEG1}, it
also holds for the subtrees $T_{1}+v,\ldots,T_{d}+v$ with roots all
at $v.$ If $d=1,$ then we are done. If $d\ge2,$ one can apply Claim
\ref{claim:3(NR-RN)_UNION} to $T_{1}+v,\ldots,T_{d}+v,$ and the
statement follows.
\end{proof}
Next, we go to the inequality (\ref{eq:3(R-R)}). We use a similar
but modified two-step method as above.
\begin{claim}
\label{claim:3(R-R)_DEG1}Let $T$ be a rooted tree with root at $v$
and $\deg(v)=1.$ Let $v_{1}$ be the closest vertex to $v$ with
$\deg(v_{1})>2.$ Let $T_{1}$ be the subtree containing $v_{1}$
obtained by removing the path $vv_{1}.$ If $T_{1}$ satisfies (\ref{eq:3(R-R)}),
then so does $T.$
\end{claim}

\begin{figure}[h]     
\begin{center}     
\begin{tikzpicture}     
	{          
	\draw[thick] (0,0)--(0.5,0); 
	\node at (0.78,0) {$\ldots$};  
	\draw[fill] (0,0) circle (0.07);     
	\draw[fill] (1.5,0) circle (0.07);
	\draw[thick] (2.1,-0.3)--(1.5,0)--(2.1,0.3);
	\draw[fill] (2.1,0.3) circle (0.07);
	\draw[fill] (2.1,-0.3) circle (0.07);
	\draw[thick] (1.1,0)--(1.5,0); 
	\draw[dashed] (1.9,1) arc (90:270:2.4cm and 1cm);
	\draw[dashed] (2.5,0.8) arc (90:270:1.5cm and 0.8cm);
	
	\node at (-0.7,0.4) {$T$};
	\node at (2.8,0.8) {$T_1$};
    \node at (0,0.3) {$v$};
	\node at (1.5,0.3) {$v_1$};
	\node at (2.2,0) {$\ldots$};  

	\node[thick,rotate=23,anchor=center] at (2.4,0.5) {$\ldots$}; 
	\node[thick,rotate=-23,anchor=center] at (2.4,-0.5) {$\ldots$}; 
	     } 	
\end{tikzpicture}\\ 
\end{center} 
\caption{$T$ and $T_1$.}
\label{fig:3(R-R)_deg1} 
\end{figure}
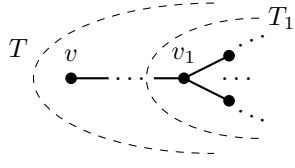 
\begin{proof}
Let $|T|-|T_{1}|=l>0.$ We have
\begin{align*}
N_{v}=N_{1}+l,\: & R_{v}=R_{1}+lN_{1}+\frac{l(l+1)}{2},
\end{align*}
and 
\begin{align*}
N_{T} & =\sum_{i=1}^{l}(N_{1}+i)+N_{T_{1}}=N_{T_{1}}+lN_{1}+\frac{l(l+1)}{2},\\
R_{T} & =\sum_{i=1}^{l}\left(R_{1}+iN_{1}+\frac{i(i+1)}{2}\right)+R_{T_{1}}\\
 & =R_{T_{1}}+lR_{1}+\frac{l(l+1)}{2}N_{1}+\frac{l(l+1)(l+2)}{6}.
\end{align*}
By assumption, we also have
\[
N_{1}N_{T_{1}}-N_{1}^{2}+N_{T_{1}}-N_{1}-3(R_{T_{1}}-R_{1})\ge0.
\]
Now
\begin{align*}
 & N_{v}N_{T}-N_{v}^{2}+N_{T}-N_{v}-3(R_{T}-R_{v})\\
= & \:l(N_{1}^{2}+N_{1}-3R_{1})+N_{1}N_{T_{1}}-N_{1}^{2}+N_{T_{1}}-N_{1}-3(R_{T_{1}}-R_{1})+lN_{T_{1}}.
\end{align*}
Applying (\ref{eq:N^2+N-3R}) to the first expression in parentheses
yields
\begin{align*}
 & N_{v}N_{T}-N_{v}^{2}+N_{T}-N_{v}-3(R_{T}-R_{v})\\
\ge & \:N_{1}N_{T_{1}}-N_{1}^{2}+N_{T_{1}}-N_{1}-3(R_{T_{1}}-R_{1})+lN_{T_{1}}\\
\ge & \:lN_{T_{1}}>0,
\end{align*}
where the last step uses the assumption. Recall that inequality (\ref{eq:N^2+N-3R}),
which will be proved later, requires $N_{1}>6.$ However, as $\deg(v_{1})>2,$
we only need to deal with $N_{1}=4$ and $N_{1}=6,$ which corresponds
to the two cases $P_{2,2}$ and $P_{2,3}.$ In the first case, $T_{1}=P_{2,2},$
and we have
\[
N_{1}=4,N_{T_{1}}=6,R_{1}=8,R_{T_{1}}=10.
\]
Then
\begin{align*}
N_{v}N_{T}-N_{v}^{2}+N_{T}-N_{v}-3(R_{T}-R_{v}) & =2l+4>0.
\end{align*}
In the second case, we have 
\[
N_{1}=6,N_{T_{1}}=10,R_{1}=15,R_{T_{1}}=20.
\]
Then
\begin{align*}
N_{v}N_{T}-N_{v}^{2}+N_{T}-N_{v}-3(R_{T}-R_{v}) & =7l+13>0.\\
\end{align*}
\end{proof}
\begin{claim}
\label{claim:3(R-R)_UNION}Let $T$ be a rooted tree with root $v$
and $\deg(v)\ge2.$ Let $T_{1}$ and $T_{2}$ be two non-singleton
subtrees also rooted at $v$ such that $T_{v}=T_{1}\overset{v}{\cup}T_{2}.$
If $T_{1}$ and $T_{2}$ both satisfy (\ref{eq:3(R-R)}), then so
does $T_{v}.$
\end{claim}

\begin{proof}
We start with
\begin{align*}
 & N_{v}=N_{1}N_{2},\:R_{v}=N_{1}R_{2}+N_{2}R_{1}-N_{1}N_{2},\\
 & N_{T}=N_{1}N_{2}+N_{T_{1}}-N_{1}+N_{T_{2}}-N_{2},\\
 & R_{T}=N_{1}R_{2}+N_{2}R_{1}-N_{1}N_{2}+R_{T_{1}}-R_{1}+R_{T_{2}}-R_{2},
\end{align*}
and the assumption yields
\begin{align*}
 & N_{1}N_{T_{1}}-N_{1}^{2}+N_{T_{1}}-N_{1}-3(R_{T_{1}}-R_{1})\ge0,\\
 & N_{2}N_{T_{2}}-N_{2}^{2}+N_{T_{2}}-N_{2}-3(R_{T_{2}}-R_{2})\ge0.
\end{align*}
Now
\begin{align*}
 & N_{v}N_{T}-N_{v}^{2}+N_{T}-N_{v}-3(R_{T}-R_{v})\\
= & \:N_{1}N_{2}(N_{1}N_{2}+N_{T_{1}}-N_{1}+N_{T_{2}}-N_{2})+N_{T_{1}}-N_{1}+N_{T_{2}}-N_{2}\\
 & -N_{1}^{2}N_{2}^{2}-3(R_{T_{1}}-R_{1}+R_{T_{2}}-R_{2}).
\end{align*}
For the last expression in parentheses, we use the two inequalities
provided by the assumption,
\begin{align*}
 & N_{v}N_{T}-N_{v}^{2}+N_{T}-N_{v}-3(R_{T}-R_{v})\\
\ge & \:N_{1}N_{2}(N_{1}N_{2}+N_{T_{1}}-N_{1}+N_{T_{2}}-N_{2})+N_{T_{1}}-N_{1}+N_{T_{2}}-N_{2}\\
 & -N_{1}^{2}N_{2}^{2}-(N_{1}N_{T_{1}}+N_{T_{1}}-N_{1}^{2}-N_{1}+N_{2}N_{T_{2}}+N_{T_{2}}-N_{2}^{2}-N_{2})\\
= & \:(N_{2}-2)(N_{1}N_{T_{1}}-N_{1}^{2})+(N_{1}-2)(N_{2}N_{T_{2}}-N_{2}^{2})\\
 & +N_{1}(N_{T_{1}}-N_{1})+N_{2}(N_{T_{2}}-N_{2})\\
\ge & \:0,
\end{align*}
as $N_{1},N_{2}\ge2$ and $N_{T_{1}}\ge N_{1},N_{T_{2}}\ge N_{2}.$
Indeed, both observations follow from the assumption that $T_{1},T_{2}$
are non-singleton subtrees.
\end{proof}
After the two claims, (\ref{eq:3(R-R)}) holds for trees with root
of degree 1.
\begin{claim}
\label{claim:3(R-R)_DEG1GENERAL}Let $T$ be a rooted tree with root
$v$ and $\deg(v)=1.$ Then $T$ satisfies (\ref{eq:3(R-R)}).
\end{claim}

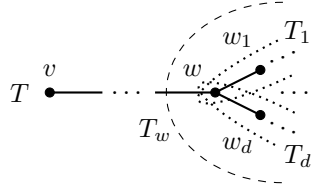
\begin{figure}[h]     
\begin{center}     
\begin{tikzpicture}     
	{          
	\draw[thick] (-0.7,0)--(0,0); 
	\node at (0.32,0) {$\ldots$};  
	\draw[fill] (-0.7,0) circle (0.06);     
	\draw[fill] (1.5,0) circle (0.06);
	\draw[thick] (2.1,-0.3)--(1.5,0)--(2.1,0.3);
	\draw[fill] (2.1,0.3) circle (0.06);
	\draw[fill] (2.1,-0.3) circle (0.06);
	\draw[thick] (0.7,0)--(1.5,0); 

	\draw[dotted,thick,rotate=26,anchor=center] (2.4,-0.4) arc (90:270:1.3cm and 0.25cm);
	\draw[dotted,thick,rotate=-27,anchor=center] (2.4,0.92) arc (90:270:1.3cm and 0.25cm);
	
	\draw[dashed] (2.4,1.2) arc (90:270:1.55cm and 1.2cm);
	
	\node at (-1.1,0) {$T$};
	\node at (0.7,-0.5) {$T_w$};
    \node at (-0.7,0.3) {$v$};
	\node at (1.2,0.3) {$w$};
	\node at (2.55,0) {$\ldots$}; 

	\node at (1.8,0.7) {$w_1$};
	\node at (1.8,-0.7) {$w_d$};

	\node[thick,rotate=23,anchor=center] at (2.4,0.47) {$\ldots$}; 
	\node at (2.6,0.8) {$T_1$};
	\node[thick,rotate=-23,anchor=center] at (2.4,-0.47) {$\ldots$}; 
	\node at (2.6,-0.85) {$T_d$};	
    } 	
\end{tikzpicture}\\ 
\end{center} 
\caption{$T_w$: dashed curve, $T_1$, ..., $T_d$: dotted curves.}
\label{fig:3(R-R)_nopath} 
\end{figure} 
\begin{proof}
We prove the statement by induction on the order. For a singleton
tree, the inequality reduces to $1\ge0.$ 

Now assume that $T$ is not a path. Let $w$ be the closest vertex
of degree greater than 2, and $T_{w}$ be the subtree with root $w.$
Denote the \foreignlanguage{american}{neighbours} of $w$ in $T_{w}$
by $w_{1},\ldots,w_{d}.$ Let $T_{i},i=1,\ldots,d,$ be the subtree
that contains $w$ and $w_{i}$ and all that is attached, as indicated
by dotted curves in Figure~\ref{fig:3(R-R)_nopath}. Note that $\deg_{T_{i}}(w)=1.$
Then, by the induction hypothesis, $T_{1},\ldots,T_{d}$ all \foreignlanguage{american}{satisfy}
(\ref{eq:3(R-R)}). So Claim \ref{claim:3(R-R)_UNION} implies that
$T_{w}$ satisfies (\ref{claim:3(R-R)_UNION}) as well. Now Claim
\ref{claim:3(R-R)_DEG1} yields the statement.

Now we deal with the case of a path. If $T$ is a path with $n$ vertices,
we have
\[
N_{v}=n,N_{T}=\frac{n(n+1)}{2},R_{v}=\frac{n(n+1)}{2},R_{T}=\frac{n(n+1)(n+2)}{6}.
\]
Hence
\begin{align*}
 & N_{v}N_{T}-N_{v}^{2}+N_{T}-N_{v}-3(R_{T}-R_{v})\\
= & \:\frac{n^{2}(n+1)}{2}-n^{2}+\frac{n(n+1)}{2}-n-3\left(\frac{n(n+1)(n+2)}{6}-\frac{n(n+1)}{2}\right)\\
= & \:0.
\end{align*}
\end{proof}
\begin{lem}
\label{lem:3(R-R)}Let $T$ be a rooted tree with root $v.$ Then
$T$ satisfies (\ref{eq:3(R-R)}).
\end{lem}

\begin{proof}
If $\deg(v)=1,$ then Claim \ref{claim:3(R-R)_DEG1GENERAL} yields
the statement. If $\deg(v)\ge2,$ then Claim \ref{claim:3(R-R)_UNION}
and \ref{claim:3(R-R)_DEG1GENERAL} together yield the statement.
\end{proof}
\begin{lem}
\label{lem:N^2+N>3R}Let $T_{v}$ be a rooted tree with root $v$
and $\deg(v)\ge2.$ Let $T_{v}\neq P_{2,2},P_{2,3}.$ Then
\[
N_{v}^{2}+N_{v}\ge3R_{v}.
\]
\end{lem}

\begin{proof}
Since $\deg(v)\ge2$ and $T_{v}\neq P_{2,2},$ $T_{v}$ can be expressed
as a union $T_{v}=T_{1}\overset{v}{\cup}T_{2},$ where both $T_{1}$
and $T_{2}$ are non-singletons. Without loss of generality, assume
$N_{1}\le N_{2}.$ We know that $N_{1}\ge2,$ $N_{2}\ge3,$ and when
$N_{1}=2,N_{2}\ge4.$

Thus, one has
\begin{align*}
N_{v}^{2}+N_{v}-3R_{v}= & \:N_{1}^{2}N_{2}^{2}+N_{1}N_{2}-3(N_{1}R_{2}+N_{2}R_{1}-N_{1}N_{2})\\
\ge & \:N_{1}^{2}N_{2}^{2}+4N_{1}N_{2}-3N_{1}\frac{N_{2}^{2}+N_{2}}{2}-3N_{2}\frac{N_{1}^{2}+N_{1}}{2}\\
= & \:N_{1}N_{2}\left((N_{1}-\frac{3}{2})(N_{2}-\frac{3}{2})-\frac{5}{4}\right),
\end{align*}
where inequality (\ref{eq:JamisonIneq}) is used in the second step.
The result is non-negative if $N_{2}\ge N_{1}>2$ or $N_{2}\ge4,N_{1}=2.$
Therefore, the inequality is proved.
\end{proof}
\begin{lem}
\label{lem:NN+N>3R}Let $T_{v}$ be a rooted tree with root $v,$
and $\deg(v)\ge2.$ Let $T_{v}\neq P_{2,2},P_{2,3}.$ Then
\[
N_{v}N_{T}+N_{T}>3R_{T}.
\]
\end{lem}

\begin{proof}
It follows immediately from Lemma \ref{lem:N^2+N>3R} that
\[
N_{v}+1\ge3\mu_{v}>3\mu_{T}.
\]
The right inequality is strict because local mean equals global mean
if and only if the tree is a trivial singleton (\cite{MR735190}).
Now, multiplying the two sides with $N_{T}$ yields the desired inequality.
\end{proof}

\section{Acknowledgement}

I would like to thank Stephan Wagner for the discussions, and Stijn
Cambie for his suggestion of strengthening inequality (\ref{eq:3(R-R)}).

\bibliographystyle{plain}
\addcontentsline{toc}{section}{\refname}\bibliography{contractionconjecture}

\end{document}